\documentclass{amsart}
\usepackage{amscd}
\usepackage{xcolor}
\newcommand{\bastb}{\bold{H}}
\newcommand{\op}[1]{\operatorname{\bold #1}}
\newcommand{\rpar}{u}
\newcommand{\xast}{x^{\ast}}
\newcommand{\expect}{\mathbb E}
\newcommand{\sa}{s_{\alpha}}
\newcommand{\xad}{x^\delta_\alpha}
\newcommand{\Cad}{C_\alpha^\delta}
\newcommand{\Vad}{V^\delta(\alpha)}
\newcommand{\N}{\mathcal N}
\newcommand{\lr}[1]{\left( #1 \right)}
\newcommand{\norm}[2]{\left\| #1 \right\|_{#2}}

\newcommand{\set}[1]{\left\{ #1 \right\}}
\newcommand{\scalar}[2]{\langle{ #1},{#2} \rangle}
\newcommand{\real}{\mathbb R}
\newcommand{\nat}{\mathbb N}
\newcommand{\af}{\mathcal S_\varphi}
\newcommand{\tr}[1]{\operatorname{tr}\left[#1\right]}
\newcommand{\range}{\mathcal R}

\newcommand{\z}{z^\delta}
\newcommand{\spc}{\mathrm{SPC}}
\newcommand{\prc}{\bold{C}}
\newcommand{\cN}{\mathcal N}
\newcommand{\xfg}[1]{X_{#1}^{\prc}}
\newcommand{\xfh}[1]{X_{#1}^{\bastb}}
\newcommand{\brf}{\bx(\alpha)}
\newcommand{\bx}{b_{\xast}}
\newcommand{\funf}{\sa}
\newcommand{\aast}{\alpha_{\ast}}
\newcommand{\bigo}{\mathcal O}
\newcommand{\peter}[1]{#1}

\newtheorem{lem}{Lemma}

\newtheorem{cor}{Corollary}
\newtheorem{prop}{Proposition}
\newtheorem{thm}{Theorem}
\theoremstyle{definition}
\newtheorem{Def}{Definition}
\newtheorem{ass}{Assumption}
\newtheorem{rem}{Remark}
\newtheorem{xmpl}{Example}

\numberwithin{equation}{section}

\title[Bayesian  inverse problems]{Bayesian inverse problems with
  non-commuting operators}  
\author{Peter Math\'e}
\email{peter.mathe@wias-berlin.de} 
\address{Weierstrass Institute, Mohrenstrasse 39, 10117 Berlin, Germany}
\begin{document}
\date{\today}
\begin{abstract}
  The Bayesian approach to ill-posed operator equations in Hilbert
  space recently
  gained attraction. In this context, and when the prior distribution
  is Gaussian, then two operators play a significant
  role, the one which governs the operator equation, and the one which
  describes the prior covariance. Typically it is assumed that these
  operators commute. Here we extend
  this analysis to non-commuting operators, replacing the
  commutativity assumption by a link condition. We discuss its
  relation to the commuting case, and we 
  indicate that this allows to use interpolation type results to
  obtain tight bounds for the contraction of the posterior Gaussian
  distribution towards the data generating element.
\end{abstract}
\maketitle

\section{Setup and Problem formulation}
\label{sec:setup}

We shall consider the equation
\begin{equation}
  \label{eq:base}
  y^{\delta} = \op K x + \delta\xi,
\end{equation}
where~$\delta>0$ prescribes a base noise level, and~$\op K\colon X\to Y$
is a compact linear operator between Hilbert spaces. The noise
element~$\xi$ is a weak random element.
If this random element~$\xi$ has covariance~$\op \Sigma$, then we
may pre-whiten Equation~(\ref{eq:base})  to get  
\begin{equation}
  \label{eq:z-model}
\z = \op \Sigma^{-1/2} \op K x + \delta\op \Sigma^{-1/2}\xi,  
\end{equation}
which is now a linear inverse problem under Gaussian white noise.

In the Bayesian framework we choose a prior for $x$. Since this is
assumed to be a tight and centered Gaussian
measure~$\cN(0,\frac{\delta^{2}}\alpha\prc)$, it is equipped with a (scaled)
covariance~$\prc$ which has a finite trace. As calculations show, the
relevant operator in the analysis will then be $\op B:= \op
\Sigma^{-1/2} \op K
\prc^{1/2}$, we refer to~\cite{AM2014} for details.

Therefore, we have (at least)  two operators to
consider, the prior covariance 
operator~$\prc$ as well as the operator~$\bastb := \op B^{\ast}\op B$.  Both
operators are 
non-negative compact  self-adjoint operators in $X$. To simplify the
analysis we shall assume that the operator~$\prc$ is
injective.

Within the present, very basic Bayesian context much is known,
and we refer to the recent survey~\cite{AM2014} and references therein. In
particular we know that the posterior is (tight) Gaussian, and we find the
following representation for the posterior mean and covariance for the
model from~(\ref{eq:z-model}):
\begin{align}\label{ali:pom-representation}
\xad& =  \prc^{1/2}\lr{\alpha \op I +\bastb}^{-1}\op B^{\ast}\z,\hfill\
      \text{(posterior mean) and}\\
\Cad& =  {\delta^{2}} \prc^{1/2} \lr{\alpha
      \op I+\bastb}^{-1}\prc^{1/2}\hfill\text{(posterior covariance).} 
                  \label{ali:cad-representation}
  \end{align}
In the study~\cite{AM2014} the authors highlight that the (square of
the) contraction of the 
posterior towards the element~$\xast$, generating the data~$\z$ is
driven by the \emph{squared posterior contraction} (SPC), given as
 \begin{equation}
    \label{eq:sqpostconc}
    \spc(\alpha,\delta):= \expect^{\xast}\expect^{\z}_{\alpha}\norm{\xast -
      x}{}^{2},\quad \alpha,\delta >0,
  \end{equation}
where the outward expectation is taken with respect to the data
generating distribution, that is, the distribution generating~$\z$
when $\xast$ is given, and the inward expectation is taken with
respect to the posterior distribution, given data~$\z$ and having
chosen a parameter~$\alpha$. Moreover, the~$\spc$ has a decomposition
  \begin{equation}
    \label{eq:base-decomposition}
    \spc(\alpha,\delta)= \bx^2(\alpha)+ \Vad + \tr{\Cad}. 
  \end{equation}
with the squared bias
$\bx^2(\alpha):= \norm{\xast - \expect^{\xast}\xad}{}^{2}$,  the estimation
variance~$\Vad:=\expect^{\xast}\norm{\xad-\expect^{\xast}\xad}{}^2$,
and the \emph{posterior spread}~$\tr{\Cad}$.
Proposition~1 ibid. asserts that the estimation variance~$\Vad$ is always
smaller than the posterior spread, thus we need to bound the
bias and posterior spread, only. Therefore we recall the form of the
bias from Lemma~1 ibid. as
\begin{equation}
  \label{eq:bias-form}
    b_{\xast}(\alpha) = \norm{\prc^{1/2} \sa(\bastb)
                        \prc^{-1/2}\xast}{},\quad \alpha >0, 
\end{equation}
where we abbreviate~$s_{\alpha}(\bastb) = \alpha\lr{\alpha\op I + \bastb}^{-1}$.
 Plainly, if $\prc$ and
$\bastb$ commute, then we have that
\begin{equation}
  \label{eq:brf-commute}
  \brf = \norm{\funf(\bastb)\xast}{}.
\end{equation}
 Also,
it is easily seen from the cyclic  commutativity of the trace that
\begin{equation}
  \label{eq:spread}
  \tr{\Cad} = {\delta^{2}} \tr{\lr{\alpha\op I + \bastb}^{-1}\prc}. 
\end{equation}
Here we aim at providing tight bounds, both for the bias and
the posterior spread for non-commuting operators~$\prc$
and~$\bastb$. For Bayesian inverse problems we are aware of only one
study~\cite{MR3084161}
.

Within the 'classical' theory of ill-posed problems this was
met earlier.
\peter{The typical situation arises, when smoothness is measured
in some (Sobolev) Hilbert scale, say e.g.~$H^{l}(\Omega)$,
where~$\Omega$ is some sufficiently smooth bounded domain, and~$l\in\real$ describes smoothness properties. 
The operator~$K$ is then linked to the scale by assuming that there is
some $\mu>0$ such that
$$
\norm{\op K x}{Y} \asymp \norm{x}{-l},\quad x\in H^{-l}(\Omega),
$$
i.e.,\ the operator~$\op K$ is smoothing with step~$l$; we highlight
such situation in~\S~\ref{sec:example}. }
Subsequently, in particular when measuring smoothness in a more general sense in
terms of variable Hilbert scales, such links were assumed in a more
general context. A comprehensive study is~\cite{MR2277542}, which
shows how interpolation in variable Hilbert scales can be used in
order to derive error bounds under such link conditions. The present
study may be seen as an application of these techniques to Bayesian
inverse problems.

We shall start in Section~\ref{sec:link-condition} with introducing the link
condition, discuss its implications, and we relate this to the commuting
case. In particular we emphasize that the assumptions made for
commuting operators yield a corresponding link condition.
Then we shall use the decomposition of the $\spc$ as
in~(\ref{eq:base-decomposition}), and thus find bounds for the bias in
Section~\ref{sec:bias-bound}, and bounds 
for the posterior spread in Section~\ref{sec:spread-bound},
respectively. We then summarize the results for giving bounds for the
squared posterior contraction in Section~\ref{sec:spc-bound}, and we
conclude with a discussion in Section~\ref{sec:conclusion}.

\section{Linking operators and scales of Hilbert spaces}
\label{sec:link-condition}

In order to introduce the fundamental link condition we need some
notions and auxiliary calculus.

\subsection{Link condition}
\label{sec:link}

We first recall the following concepts from~\cite{AM2014,MR2992966}.

\begin{Def}
  [index function] A function~$\psi\colon (0,\infty)\to \real^{+}$ is
  called an \emph{index function} if it is a continuous non-decreasing
function with $\psi(0)=0.$ 
\end{Def}
\begin{Def}[partial ordering for index functions]
   Given two index functions
$g, h$ we shall write $g\prec h$ \peter{($h$ is beyond~$g$)}, if the function~$t\mapsto h(t)/g(t)$ is
an index function ($h$ tends to zero faster than $g$). 
\end{Def}
The link condition which we are going to introduce now will be based
upon a partial ordering of self-adjoint operators, and we refer
to~\cite{MR1477662} for a comprehensive account. Although the
monograph formally treats matrices, only, most of the results transfer
to (bounded compact) operators in Hilbert space.
\begin{Def}
  [partial ordering for self-adjoint operators]
Let $\op G$ and $\op G^{\prime}$ be boun\-ded non-negative self-adjoint
operators in some Hilbert spaces $X$. We say that $\op G\leq \op G^{\prime}$ if for all
$x\in X$ the inequality $\scalar{\op G x}{x}\leq \scalar{\op G^{\prime}
  x}{x}$ holds true.
\end{Def}
The following concept of 'concavity' is the extension of concavity
from real functions to self-adjoint operators by functional calculus.
\begin{Def} Let $f\colon [0,a]\to \real^+$ be a continuous function.
 It is called {\em operator concave}\/ if for any pair $\op G,\op G^{\prime}\geq
    0$ of self-adjoint operators with spectra in $[0,a]$ we have
    \begin{equation}
      \label{eq:op-conc}
      f(\frac{\op G + \op G^{\prime}}{2})\geq \frac{f(\op G) + f(\op G^{\prime})}{2}.
    \end{equation}
\end{Def} 
We mention that operator concave function must be operator monotone,
i.e., if~$\op G\leq \op G^{\prime}$ then we will have that~$f(\op G)
\leq f(\op G^{\prime})$,
see~\cite[Thm.~V.2.5]{MR1477662}\footnote{Formally, operator monotone
  functions are defined on~$[0,\infty)$. The asserted monotonicity
can be seen from the proof of Theorem~V.2.5 ibid. }.
The concept of operator concavity will be crucial for the interpolation, below. However,
the above  partial ordering has also implications for the ranges of the
operators, and this is comprised in
\begin{thm}
[{Douglas' Range Inclusion
Theorem, see~\cite{MR0203464}}]\label{thm:douglas}
Let us onsider operators~$\op S\colon Y \to X$ and $\op T\colon Z \to X$, acting between
Hilbert spaces. The following assertions are equivalent.
\begin{enumerate}
\item $\range(\op S) \subset \range(\op T),$
\item there is a constant~$C$ such that~$\op S \op S^\ast \leq
  C^{2}\op T \op T^\ast$,
\item there is a constant~$C$ such that~$\norm{\op S^{\ast}x}{Y} \leq C
  \norm{\op T^{\ast}x}{Z}\ x\in X$, \label{enum:norm-ineq}
\item there is a 'factor'~$\op R\colon Y\to Z,\ \norm{\op R}{}\leq C$,
  such that~$\op S = \op T \op R$.
\end{enumerate}
\end{thm}
Of course, for self-adjoint operators~$\op S,\op T\colon X\to X$ the norm
estimate from~(\ref{enum:norm-ineq}) is again for $\op S=\op S^{\ast}$
and~$\op T=\op T^{\ast}$. Also, if the operator~$\op T$ is injective then the
composition~$\op T^{-1}\op S$ is a bounded operator and~$\norm{\op
  T^{-1}\op S}{} =
\norm{\op R}{}\leq C$.

As it was stressed above, the governing operators~$\prc$ and~$\bastb =
\op B^{\ast} \op B$ are non-negative self-adjoint. As an immediate
application, by considering~$\op B\colon X \to Y$ and its self-adjoint
analog~$\op B^{\ast}\op B\colon X \to X$ we plainly have that
$$
\norm{\op Bu}{X} =
\norm{\lr{\op B^{\ast}\op B}^{1/2} u}{X},\quad u\in~X,
$$ such that~$\range(\op B^{\ast})
= \range(\lr{\op B^{\ast}\op B}^{1/2})= \range(\bastb^{1/2})$.
\medskip

Before formally introducing the link assumption, we first make the standing assumption that the compound
mapping~$\op \Sigma^{-1/2}\op K\colon X\to Y$ is bounded. %
The link condition will
provide us with a 'tuning index function'~$\psi$ such that the
ranges of~$\psi(\prc)$ and~$\op K^{\ast}\op \Sigma^{-1/2}$ coincide, and hence
we shall assume that
$$
\norm{\psi(\prc)v}{X} \asymp \norm{\op \Sigma^{-1/2}\op K v}{Y},\quad
v\in X.
$$
Using this with~$v:= \prc^{1/2}u,\ u\in X$ we arrive at
$$
\norm{\Theta(\prc)u}{X} \asymp
\norm{\op \Sigma^{-1/2}\op K \prc^{1/2}u}{Y} = \norm{\bastb^{1/2}u}{Y},\quad
u\in X,
$$
where we introduced the function
\begin{equation}\label{eq:theta}
\Theta(t) = \Theta_{\psi}(t):= \sqrt t \psi(t),\quad t>0,
\end{equation} 
and the operator~$\bastb$ \peter{is given} as before from~$\op B
:= \op \Sigma^{-1/2}\op K \prc^{1/2}$ as~$\bastb = \op B^{\ast}\op B$.
We
observe that its square~$\Theta^{2}$ is strictly monotone, and it increases super-linearly. Below,
the inverse will play an important role, and we stress that this will
be a sub-linearly increasing index function, as this typically is the case for power
type functions~$g(t) := t^{q}$ with~$0 < q \leq 1$. So, we  formally
make the following
\begin{ass}
  [link condition]\label{ass:link}
There are an index function~$\psi$, and constants
$0< m \leq 1 \leq M < \infty$ such that
\begin{equation}
      \label{eq:link-psi}
     m \norm{\psi(\prc) u}{} \leq \norm{
      \op \Sigma^{-1/2}\op K  u}{} \leq  M \norm{\psi(\prc) u}{},\quad u\in X.   
    \end{equation}
Moreover, with the function~$\Theta$  from~(\ref{eq:theta}), \peter{we assume
that} the related function
\begin{equation}
  \label{eq:f-interpolation}
   f_{0}(s) := \lr{\lr{\Theta^{2}}^{-1}(s)}^{1/2},\quad s>0,
\end{equation}
has an  \emph{operator concave square~$f_{0}^{2}$}.
\end{ass}

We first draw the following consequence. For this we let
\begin{equation}
  \label{eq:varphi0}
  \varphi_{0}(t) := \sqrt t,\quad t>0,
\end{equation}
throughout this study.
\begin{prop}\label{prop:C0}
  Under Assumption~\ref{ass:link} we have that~$\range(\prc^{1/2}) =
  \range(f_{0}(\bastb))$. 
In particular the operator~$f_{0}(\bastb) \varphi_{0}(\prc)^{-1}$ \footnote{\peter{We agree
  upon the following convention. For an index function, say~$s\mapsto
  f(s)$ the symbol~$f^{-1}$ denotes the inverse function, whereas for
  a related operator~$f(\op G)$ the symbol~$f(\op G)^{-1}$ denotes the
  inverse operator, corresponding to the reciprocal function, i.e.,\
  $f(\op G)^{-1} = \lr{\frac 1 f}(\op G)$. There is a little
  ambiguity, but the precise meaning will be clear form the context.}}is norm
bounded by~$M$.
\end{prop}
\begin{proof}
Arguing as above, the inequalities in~(\ref{eq:link-psi}) have their 
counterpart for the function~$\Theta$ as
\begin{equation}
      \label{eq:link-theta}
     m \norm{\Theta(\prc) u}{} \leq \norm{
       \bastb^{1/2} u}{} \leq  M \norm{\Theta(\prc) u}{},\quad u\in X.   
    \end{equation}
We can rewrite the left hand side as
$$
\Theta^{2}(\prc) \leq \frac 1 {m^{2}} \bastb. 
$$
Since $f_0^2$ is assumed
to be operator concave, and hence operator monotone we conclude that
\begin{equation}
  \label{eq:prc-upper-bound}
\prc \leq f_{0}^{2}\lr{ \frac 1 {m^{2}} \bastb}\leq \frac 1 {m^{2}}  f_{0}^{2}\lr{\bastb},
\end{equation}
where we used that~$0< m \leq 1$.
Rewriting this in terms of a norm inequality shows
that
\begin{equation}
  \label{eq:1m-bound}
\norm{\prc^{1/2}u}{X} \leq \frac 1 m \norm{f_{0}(\bastb)u}{X},\
u\in X,  
\end{equation}
 and hence that~$\range(\prc^{1/2})\subseteq
\range(f_{0}(\bastb))$. The other inclusion is proven similar by using
the right hand side of~(\ref{eq:link-theta}), and hence it is
omitted. The norm boundedness is a consequence of
Theorem~\ref{thm:douglas}, as it was stressed after its formulation.
\end{proof}

 Basically, the above inequalities in~(\ref{eq:link-theta}) correspond to Assumption~3.1 (2 \& 3)
from~\cite{MR3084161}, see Section~6, ibid., if the function~$\Theta$
is a power function. 

\subsection{Linking commuting operators}
\label{sec:commute}

In previous studies, dealing with commuting operators, no functional
dependence was assumed, except the recent survey~\cite{AM2014}.
\peter{As it will be shown next} the present setup of a link condition extends previous
studies restricted to commuting operators. The calculus with using
functional dependence instead of 
asymptotic behavior of singular numbers seems simpler to handle.

We
start with the following technical assertion.
\begin{lem}
Suppose that we have commuting self-adjoint non-negative compact
operators~$\op G, \op G^\prime$
in Hilbert space. If the operator~$\op G$ has only simple eigenvalues then there is a continuous function~$\psi\colon
[0,\norm{\op G}{}]\to \real^{+}$, with~$\lim_{u\to 0}\psi(u) = 0$,
such that~$\psi(\op G)
= \op G^{\prime}$.
\end{lem}
\begin{proof}
Indeed, the pair~$\op G,\op G^{\prime}$ is commonly diagonalizable, and we may consider
(infinite) diagonal matrices~$D_{s}$ and~$D_{t}$, having the
eigenvalues of $\op G$ and~$\op G^{\prime}$ on the diagonals. Since all eigenvalues
of~$\op G$ were assumed to be simple, we can assume
that~$s_{1} > s_{2} > \dots >0$. The corresponding eigenvalues
for~$\op G^{\prime}$
will not be ordered, in general. We consider the mappings~$s,t\colon
\nat\to \real^{+}$ assigning~$s(j) = s_{j}$ and~$t(j) := t_{j}$,
respectively. The mapping~$s$ is strictly decreasing, such that we may
consider the composition~$\bar \psi:= t\circ s^{-1}\colon \lr{s_{j}} \mapsto
\lr{t_{j}}$. By linear interpolation this extends to a continuous
mapping~$\psi\colon (0,\norm{\op G}{}]\to \real^{+}$. We need to show
that~$\psi(u)\to 0$ as~$u\to 0$. But, since the sequence~$t_{j},\
j=1,2,\dots$ has zero as only accumulation point, there is~$N\in\nat$
with~$t_{n}\leq \varepsilon$ for~$n\geq N$. Let~$\delta:=
s_{N}>0$. Then for~$n\geq N$ we find that~$\psi(s_{n}) = t_{n}\leq
\varepsilon$, and by linear interpolation this extends to the whole
interval~$[0,\delta]$.
\end{proof}

Of course, we cannot find that the above function~$\psi$ be an index
function. For this to hold additional assumptions need to be
made. Here we consider the situation as it was assumed
in~\cite{MR2906881}, cf.~ Assumption~3.1 ibid.

\begin{prop}
  Suppose that with respect to the common eigenbasis~$e_{j},\
  j=1,2,\dots$ the 
corresponding eigenvalues~$s_{j}$ of the prior covariance, and $t_{j}$ of
the operator~$\op \Sigma^{-1/2}\op K$ obey some asymptotic behavior, say in
the power type case~$s_{j} = j^{-(1+2a)}$ and~$\bar m j^{-p} \leq
t_{j}\leq \bar M j^{-p}$ for $j\in\nat$ and parameters~$a,p>0$. 
Then Assumption~\ref{ass:link} holds for the (index) function~$\psi(t)
= \tfrac{\bar m + \bar M}2 t^{p/(1 + 2a)},\ t>0$, and with constants~$m:= 2\bar m/(\bar m +
\bar M)$, and~$M:= 2\bar M/(\bar m + \bar M)$.
\end{prop}
\begin{proof}
First, by construction we find that~$0 < m \leq 1\leq M <\infty$. Also
we see that~$m\psi(t) = \bar m t^{p/(1 + 2a)}$. Thus,  for~$u =
\sum_{j=1}^{\infty} u_{j} e_{j}\in X$ we bound
\begin{align*} 
  m^{2} \norm{\psi(\prc)u}{}^{2} 
&= m^{2} \sum_{j=1}^{\infty} \psi^{2}(s_{j}) u_{j}^{2} = \bar m^{2}
  \sum_{j=1}^{\infty} j^{-2p}u_{j}^{2} \\ 
&\leq \sum_{j=1}^{\infty} t_{j}^{2}u_{j}^{2} = \norm{\op \Sigma^{-1/2}\op K u}{}^{2}.
\end{align*}
The other inequality is proven similarly, and we omit the
proof. Finally, the function~$f_{0}^{2}(s) \propto s^{\frac{1 + 2a}{1 + 2a +
  2p}}$ is operator concave, completing the proof. 
\end{proof}

Therefore, in order to have a fair comparison, if $\prc$ and $\bastb$
were commuting we would instead assume that 
$\Theta^{2}(\prc) =\bastb $. This should be kept in mind when
comparing the subsequent bounds.

\subsection{Prototypical example}
\label{sec:example}

\peter{The following example was first presented in~\cite{MR0488721} when considering
projection schemes in Hilbert scales. Let~$\Omega\subset \real^{2}$ be a
bounded, sufficiently smooth domain. Let~$H^{l}(\Omega)$ denote the
corresponding Sobolev spaces of order~$l\geq 0$, and for $l<0$ we
let~$H^{l}(\Omega) := \lr{H^{l}(\Omega)}^{\prime}$, the adjoint
space. We consider the Radon transform given as follows. Let~$Z:=
\set{(\omega,t),\ \omega\in\real^{2}, \norm{\omega}{}=1,\ t\in\real
}$. For given~$(\omega,t)$ we consider the line~$\set{u\in\real^{2},\
  \scalar u \omega = t}$, endowed with corresponding Lebesgue
measure~$\tau_{(\omega,t)}$. The Radon transform~$\op K\colon
L_{2}(\Omega) \to L_{2}(Z)$ is then given as 
$$
(\op K x)(\omega,t) := \int x(s)\; d\tau_{(\omega,t)}(s),\quad
(\omega,t)\in Z.
$$
It was shown in~\cite{MR0519587} that this operator obeys
$$
\norm{\op K x}{Y} \asymp \norm{x}{H^{-1/2}(\Omega)},\quad x\in H^{-1/2}(\Omega).
$$
Since the natural embedding~$L_{2}(\Omega)\hookrightarrow
H^{-1/2}(\Omega)$ is compact, there is a compact, bounded self-adjoint
operator~$\op G$ with~$\norm{\op G^{1/2}x}{L_{2}(\Omega)} =
\norm{x}{H^{-1/2}(\Omega)}$, and we refer to the construction
of operators generating Hilbert scales in~\cite[Capt.~IV,
\S~1.10]{MR649411}. Overall we arrive at
$$
\norm{\op K x}{Y} \asymp \norm{\op G^{1/2}x}{L_{2}(\Omega)},\quad x\in L_{2}(\Omega),
$$
which is a specific form of the link condition in
Assumption~\ref{ass:link}, when~$\op \Sigma =\op I$, and the
covariance operator~$\prc$ is a power of~$\op G$ in order to be of
trace class.
}

\subsection{Variable Hilbert scales and their interpolation}
\label{sec:vhs}

We recall the concept of variable
Hilbert scales.
Given an injective positive self-adjoint operator~$\op G$, and some index
function~$f$ we equip $\range(f(\op G))$ with the norm~$\norm{x}{f} =
\norm{w}{}$, where the element ~$w$ is (uniquely) obtained from~$x=
f(\op G) w$. This makes $\lr{\range(f(\op G)),\norm{\cdot}{f}}$ a Hilbert
space. Since this can be done for any index function~$f$ we agree to
denote the resulting spaces by $X_{f}^{\op G }$. 

\peter{Below} we shall consider the scales generated by $\prc$,
i.e.,\ $\xfg \varphi$, and by $\bastb$, hence the spaces~$\xfh f$ (We
shall reserve Greek letters for index functions related to the
scale~$\xfg \varphi$). 

We shall use interpolation of operators in variable Hilbert scales,
and we recall the fundamental result from~\cite{MR2277542}.
\begin{thm}
  [Interpolation theorem,
  {\cite[Thm.~5]{MR2277542}}]\label{thm:interpol} Let  $\op G,\op G^{\prime}\geq 0$ be self-adjoint operators with spectra in $[0,b]$
 and $[0,a]$, respectively. 
  Furthermore, let $\varphi,\rho$ and $r$ be index functions ($\rho$
  strictly increasing) on intervals
  $[0,b]$ and $[0,a]$, respectively, such that $b\geq \norm{\op G}{}$ and
  $\rho(b)\geq r(a)$. 
Then the function 
\begin{equation*}
f(t):= \varphi(\rho^{-1}(r(t))),\quad 0<t\leq a,  
\end{equation*}
is well defined.
The following  assertion holds true:
If  $t\to
  \varphi^2((\rho^2)^{-1}(t))$ is operator concave on $[0,\rho^2(b)]$ then
\begin{equation}
   \norm{\op S x}{}\leq C_1 \norm{x}{},\quad x\in X, \label{ali:cor1} 
\end{equation}
and
\begin{equation}
  \norm{\rho(\op G)\op S x}{}\leq C_2 \norm{r(\op G^{\prime})x}{},\quad x\in X,  \label{ali:cor2} 
\end{equation}
yield
\begin{equation}
 \norm{\varphi(\op G)\op S x}{}\leq \max\set{C_1,C_2} \norm{f(\op G^{\prime})x}{},\quad x\in X .  \label{ali:cor3}  
\end{equation}
\end{thm}
 We depict the
interpolation setup in Figure~\ref{fig:interp}.
\begin{figure}
\center \fbox{\parbox{0.95\textwidth}{
  \centering
 \begin{equation*}   
\begin{CD}
    \op G : @. X_{\rho}^{\op G} @>J_{\rho}>> X_{\varphi}^{\op G}
    @> J_{\varphi}>> X\\ 
    @.  @VV \op S^{\ast} V @VV \op S^{\ast}  V @VV \op S^{\ast} V \\
    \op G^{\prime}: @. X_{r}^{\op G^{\prime}} @>J_{r}>>  X_{f }^{\op G^{\prime}} @>J_{f}>> X
  \end{CD}
\end{equation*} 
\caption{The setup of interpolation. \peter{The mappings~$\op J_{\circ}$
  denote the canonical embeddings.}
  The position of
    $X_{\varphi}^{\op G}$
    between $X_{\rho}^{\op G}$ and $X$ is 
given by the function $t\to \varphi^{2}((\rho^{2})^{-1}(t))$,  and $f$ is
determined in such a way that $X_{f }^{\op G^{\prime}}$ has the appropriate position in
the scale on bottom.}
  \label{fig:interp}
}}
\end{figure}
\begin{rem}\label{rem:inter}
We mention the following important fact. The conditions on the
operator~$\op S$ in Theorem~\ref{thm:interpol} correspond to the
Figure~\ref{fig:interp} with $\op S^{\ast}$, the adjoint operator. 
 We refer to~\cite[Cor.~2]{MR2277542} for details.

  It is important to notice that, in contrast to the commuting case no
  link between the scales $X_{\varphi}^{\op G}$ and $X_{f }^{\op G^{\prime}}$ can be
  established whenever $\varphi$ is beyond $\rho$ ($\rho \prec \varphi$), or $f$ is beyond
  $r$ ($r \prec f$).
\end{rem}
For the above interpolation it is crucial that the space~$X^{\op G}_{\varphi}
$ is intermediate between $X^{\op G}_{\rho}$ and $X$, i.e.,\ we have
continuous embeddings as in Figure~\ref{fig:interp}. The position is
described as follows.
\begin{Def}
  [position of an intermediate Hilbert space]\label{def:position}
The position of $X_{\varphi}^{\op G}$
    between $X_{\rho}^{\op G}$ and $X$ is 
given by the function $t\to \varphi^{2}((\rho^{2})^{-1}(t))$.
\end{Def}
\begin{rem}
  If we imagine that the spaces~$X^{\op G}_{\rho}$ and $X^{\op G}_{\varphi}$
  were Sobolev Hilbert spaces with smoothness $0< f \leq r$ then the
  position would be the quotient~$f/r\leq1$. In the present context
  this corresponds to the power type function~$t \mapsto t^{f/r}$
  which is concave, even operator concave, and we have seen in
  Theorem~\ref{thm:interpol} that operator concavity is essential for
  establishing results on operator interpolation.
\end{rem}

\section{Bounding the bias}
\label{sec:bias-bound}

We shall use the interpolation result with $\op G:= \prc$ and $\op
G^{\prime} = \bastb$ from above, and
for various index functions and operators~$\op S$. 
We recall the description of the bias in the
decomposition~(\ref{eq:base-decomposition}) given
in~(\ref{eq:bias-form}) as
$$
\bx(\alpha) =  \norm{\prc^{1/2} \sa(\bastb)\prc^{-1/2}\xast}.
$$
To proceed we assign smoothness to the data generating element~$\xast$
\emph{relative to the covariance operator~$\prc$}.
\begin{ass}
  [source set]\label{ass:smoothness}
There is an index function~$\varphi$ such that
$$
\xast \in\af :=  \set{x,\quad x= \varphi(\prc)v,\ \norm{v}{}\leq 1}.
$$
\end{ass}

\peter{Using the
estimate~(\ref{eq:1m-bound}) from the proof of}
Proposition~\ref{prop:C0},  we can bound for an element~$\xast$ \peter{which
  obeys Assumption~\ref{ass:smoothness}} the bias~$\brf$ as
\begin{align*}
  \brf & \leq \frac 1 m \norm{f_{0}(\bastb) \funf(\bastb)
    \varphi_{0}(\prc)^{-1} \varphi(\prc)}{}\\
& = \frac 1 m \norm{ \funf(\bastb)f_{0}(\bastb)
    \varphi_{0}(\prc)^{-1} \varphi(\prc)}{}.
\end{align*}
Again, we emphasize that the intermediate operator~$f_{0}(\bastb)
    \varphi_{0}(\prc)^{-1}$ is bounded in norm, such that the right hand
    side is finite.

In our subsequent analysis we shall distinguish three cases. These are
determined by the relation of the given function~$\varphi$ with
respect to the function~$\varphi_{0}$ and the benchmark~$\Theta$. 
%
\peter{These cases
    are not exhaustive, i.e.,\ there are index functions~$\varphi$ for
    which neither of the cases applies.}
\begin{figure}[ht]
\center
\fbox{\parbox{0.95\textwidth}{
  \centering
 \begin{description}
\item[regular] This  case is obtained when $\varphi_{0}(\prc)^{-1}
\varphi(\prc)$ is a bounded operator, and hence when $\varphi_{0}\prec
\varphi\prec  \Theta$.
\item[low-order] When $1 \prec \varphi\prec \varphi_{0}$ we speak of
  the low-order case.
\item[high-order] If $\varphi$ is
beyond the benchmark $\Theta$ ($\Theta\prec \varphi$)
then we call this the high-order case. We shall need
additional assumptions to treat this.
\end{description} 
  \caption{Cases which are considered for interpolation.}
  \label{fig:cases}
}}
\end{figure}
We turn to the detailed analysis of these cases.

\subsection{Regular case:  $\varphi_{0}\prec \varphi \prec \Theta$} 
\label{sec:regular}

In this case the operator~$\prc^{-1/2} \varphi(\prc) =
\varphi_{0}(\prc)^{-1} \varphi(\prc)$ is a bounded 
self-adjoint operator. The position of the Hilbert space
$ \xfg \Theta \subset X_{\varphi/\varphi_{0}}^{\prc}\subset X$ is then
given through
\begin{equation}
  \label{eq:g-interpolation}
g^{2}(t) := \lr{\frac \varphi{\varphi_{0}}}^{2}\lr{\lr{\Theta^{2}}^{-1}(t)},\quad t>0. 
\end{equation}
\begin{prop}\label{pro:regular-bound}
  Suppose that $\varphi_{0}\prec \varphi \prec
  \Theta$, and that the function~$g^{2}$ is operator
  concave. Under Assumptions~\ref{ass:link} \&~\ref{ass:smoothness} we have
  that
$$
\brf  \leq \frac M m \norm{\funf(\bastb) \varphi\lr{f_{0}^{2}(\bastb)}}{}.
$$
\end{prop}
\begin{proof}
Under the assumptions made we apply Theorem~\ref{thm:interpol} with
$S=\op I$, the identity operator, to see that
$$
\norm{\lr{\frac{\varphi}{\varphi_{0}}}(\prc)u}{}\leq M \norm{g(\bastb) u}{} 
,\quad u\in X.
$$
By Theorem~\ref{thm:douglas} this means that for any $u,\ \norm{u}{}\leq 1$ we find that
$\bar u,\ \norm{\bar u}{}\leq M$ with
$\frac{\varphi}{\varphi_{0}}(\prc) u = g(\bastb) \bar u$.
Hence we can bound
\begin{align*}
  \norm{\funf(\bastb) f_{0}(\bastb) \lr{\frac{\varphi}{\varphi_{0}}}(\prc)
  u}{}
& = \norm{\funf(\bastb) f_{0}(\bastb) g(\bastb)\bar u}{}\\
&\leq M \norm{\funf(\bastb) f_{0}(\bastb) g(\bastb)}{}.
\end{align*}
But, we check that
$$
f_{0}(t) g(t) = f_{0}(t) \frac{\varphi(f_{0}^{2}(t))}{f_{0}(t)} =
\varphi(f_{0}^{2}(t)),\quad t>0,
$$
which gives 
$$
 \norm{\funf(\bastb) f_{0}(\bastb) \lr{\frac{\varphi}{\varphi_{0}}}(\prc)
  u}{} \leq M  \norm{\funf(\bastb) \varphi\lr{f_{0}^{2}(\bastb)}}{}.
$$
Since this holds for arbitrary $u\in X,\ \norm{u}{}\leq 1$ we conclude
that
$$
\brf \leq \frac 1 m \norm{\funf(\bastb) f_{0}(\bastb) \lr{\frac{\varphi}{\varphi_{0}}}(\prc)
  u}{}\leq  \frac M m \norm{\funf(\bastb) \varphi\lr{f_{0}^{2}(\bastb)}}{},
$$
and this completes the proof in the regular case.
\end{proof}

\subsection{Low-order case: $1 \prec \varphi \prec \varphi_{0}$}
\label{sec:low}
For the application of Theorem~\ref{thm:interpol}  we
recall the definition of the operator~$\op S:= f_{0}(\bastb)
\varphi_{0}(\prc)^{-1}\colon X \to X$.
\begin{prop}\label{pro:low-order}
Suppose that the Assumptions~\ref{ass:link} \&~\ref{ass:smoothness}
hold, and that $1 \prec \varphi \prec \varphi_{0}$. Then the position
  of $\xfg \varphi$ between $\xfg {\varphi_{0}}$ and $X$ is given by
  the function~$\varphi^{2}$. If $\varphi^{2}$ is operator concave
  then
$$
\brf \leq \frac M m \norm{\funf(\bastb)\varphi\lr{f_{0}^{2}(\bastb)}}{}.
$$
\end{prop}
\begin{proof}
We aim at applying Theorem~\ref{thm:interpol} for the
operator~$\op S^{\ast}$. From Proposition~\ref{prop:C0} we know
that~$\norm{\op S^{\ast}\colon X \to X}{}\leq M$. But also we see that
$$
\norm{\varphi_{0}(\prc) \op S^{\ast}u}{} =
\norm{\varphi_{0}(\prc)\varphi_{0}(\prc)^{-1} f_{0}(\bastb)u}{} =
\norm{f_{0}(\bastb)u}{} ,\ u\in X. 
$$
The position of~$\xfg \varphi$ between $\xfg {\varphi_{0}}$ and~$X$ is
given by
$$
\varphi^{2}\lr{\lr{\varphi_{0}^{2}}^{-1}(t)} = \varphi^{2}(t),\quad t>0,
$$
and this was assumed to be operator concave.
  Thus Theorem~\ref{thm:interpol} applies and yields
$\norm{\varphi(\prc) \op S^{\ast} u}{} \leq M \norm{g(\bastb)u}{},\ u\in
X$ for the function~$g$ given from
$$
g(t) := \varphi\lr{\varphi_{0}^{-1}(f_{0}(t))} =
\varphi\lr{f_{0}^{2}(t)},\quad t>0.
$$
By virtue of Theorem~\ref{thm:douglas} we
find that for every $x= \op S \varphi(\prc)w$ there is $\bar w,\ \norm{\bar w}{}\leq M$ with
$$
\op S \varphi(\prc)w  = g(\bastb) \bar w.
$$ 
We conclude that therefore
\begin{align*}
 \brf &\leq \frac 1 m \norm{\funf(\bastb) S \varphi(\prc)}{}
\leq  \frac M m \norm{\funf(\bastb)g(\bastb)}{}\\
 &=  \frac M m \norm{\funf(\bastb)\varphi\lr{f_{0}^{2}(\bastb)}}{}
\end{align*}
The proof is complete.
\end{proof}

\subsection{High-order case: $\Theta \prec \varphi$}
\label{high}

If we want to extend the results to smoothness beyond $\Theta$
then we need to assume a link condition at a later position than
$\bastb^{1/2} $. Therefore, we shall impose the following lifting condition.
\begin{ass}
  [lifting condition]\label{ass:lifting}
There are  some $\rpar >1$ and constants $m \leq 1 \leq M$ such that
\begin{equation}
      \label{eq:lift-r}
     m^{\rpar } \norm{\Theta^{\rpar}(\prc) u}{} \leq \norm{
       \bastb^{\rpar/2} u}{} =  M^{\rpar} \norm{\Theta^{\rpar}(\prc) u}{},\quad u\in X.   
    \end{equation}
\end{ass}
This is actually stronger than the original link condition from
Assumption~\ref{ass:link}. Indeed, as the Loewner--Heinz
Inequality asserts, for~$\rpar>1$ the function~$t\mapsto t^{1/\rpar}$ is operator
monotone, see~\cite[Thm.~V.1.9]{MR1477662}, or~\cite[Prop.~8.21]{MR1408680}
we have 
\begin{cor}\label{cor:lifting}
  Assumption~\ref{ass:lifting} yields Assumption~\ref{ass:link}.
\end{cor}
\begin{rem}\label{rem:lifting}
  We stress that in case of commuting operators~$\prc$ and $\bastb$ the
  assumptions~\ref{ass:lifting} and~\ref{ass:link} are equivalent,
  i.e., the link condition  implies the lifting condition. This can
  be seen using the Gelfand-Naimark theorem, and we refer
  to~\cite[Prop.~8.1]{MR2213075} for a similar assertion with detailed
  proof.
\end{rem}

Again we shall deal with the smoothness $\varphi_{0}(\prc)^{-1}\varphi(\prc)$,
and, as in the regular case,  we ask for the position of the corresponding space between $\xfg
{\peter{\Theta^{\rpar}}}$ and $X$.
This gives the following result, extending the cases of regular and
low smoothness, however, under additional requirement on the link.
\begin{prop}
  \label{pro:high-smoothness}
Suppose that the Assumptions~\ref{ass:lifting} \&~\ref{ass:smoothness}
hold,  and that $\varphi/\varphi_{0} \prec
  \Theta^{\rpar}$. Assume that for~$g$ from~(\ref{eq:g-interpolation}) the
  function~$t\mapsto g^{2}(t^{1/\rpar})$ is operator 
  concave. Then we have
$$
\brf \leq \frac M m \norm{\funf(\bastb) \varphi\lr{f_{0}^{2}(\bastb)}}{}.
$$
\end{prop}
\begin{proof}
  The proof is similar to the regular case.  The function~$t\mapsto
  g^{2}(t^{1/\rpar})$ is exactly 
the position of $\xfg {\varphi/\varphi_{0}}$ between $\xfg
{\Theta^{\rpar}}$ and $X$ is given as
$$
g_{\rpar}^{2}(t) := \lr{\frac{\varphi}{\varphi_{0}}}^{2}\lr{\lr{\Theta^{2\rpar}}^{-1}(t)}.
$$
It is readily checked that~$\lr{\Theta^{2\rpar}}^{-1}(t) =\lr{\Theta^{2}}^{-1}(t^{1/\rpar}) $.
Thus we find that~$g_{\rpar}^{2}(t) = g^{2}(t^{1/\rpar})$, and this is assumed to be
operator concave, such that we can use the Interpolation
Theorem~\ref{thm:interpol}, and we find that
$$
\norm{ \frac{\varphi}{\varphi_{0}}(\prc)u}{} \leq M
\norm{f(\bastb^{\rpar/2})u}{},\quad u\in X,
$$
where the function~$f$ is given as
$$
f(t) :=\frac{\varphi}{\varphi_{0}}\lr{\lr{\Theta^{\rpar}}^{-1}(t)} =
\frac{\varphi}{\varphi_{0}}\lr{\Theta^{-1}(t^{1/\rpar})},\quad t>0. 
$$
This yields that~$f(\bastb^{\rpar/2}) =
\frac{\varphi}{\varphi_{0}}\lr{\Theta^{-1}(\bastb^{1/2})}$. 
 Now, as in the regular case,  we arrive at 
$$
\brf\leq \frac M m \norm{\funf(\bastb) f_{0}(\bastb) g(\bastb)}{}.
$$
We have seen there that $f_{0}(t) g(t) = \varphi(f_{0}^{2}(t))$,
which completes the proof.
\end{proof}

\begin{rem}
  The following comment seems interesting. In the high-order case, the
  function~$g^{2}$ will in general not  be operator concave. However,
  the assumption which is made above, says that by re-scaling this
  will eventually be operator concave if the scaling factor $\rpar$ is
  large enough, see the discussion at the end of this section.  Of
  course this does not mean that the lifting 
  Assumption~\ref{ass:lifting} will hold automatically. This is still a
  non-trivial assumption.
\end{rem}

\subsection{Saturation}
\label{sec:saturation}

In all the above cases, the low-order, regular and the high-order
one, we were able to derive a bias bound as
in Proposition~\ref{pro:high-smoothness}, albeit under case specific
assumptions. This bound cannot decay arbitrarily fast, and this is
known as \emph{saturation} in the regularization theory, see
again~\cite{AM2014}. Indeed, the maximal decay rate, as~$\alpha\to 0$
is linear, unless~$\xast =0$, which is a result of the structure of the
term~$\sa(\bastb)$. This maximal decay rate is
achieved when~$\varphi\lr{f_{0}^{2}(t)} \asymp t$, which  means that~$\varphi(t) \asymp \Theta^{2}(t)$. Thus the
maximal smoothness for which optimal decay of the bias can be achieved
is given by the index function~$\Theta^{2}$. This yields
the following important remark, specific for non-commuting operators.
\begin{rem}\label{rem:saturation}
Suppose that smoothness is given as in
Assumption~\ref{ass:smoothness} with an index function~$\varphi$, and
that we find the function~$\Theta$ as in~(\ref{eq:theta}).
  Within the range~$0 \prec \varphi\prec \Theta$ (low-order and
  regular cases) the link condition, Assumption~\ref{ass:link},
  suffices to yield optimal order decay of the bias. However, within
  the range~$\Theta \prec \varphi \prec \Theta^{2}$ the lifting, as
  given in Assumption~\ref{ass:lifting} cannot be avoided. This effect
  cannot be seen for commuting operators~$\prc$ and~$\bastb$, because
  there the lifting is equivalent to the
  original link condition, as discussed in
  Remark~\ref{rem:lifting}. We also observe that within the present
  context, the lifting to~$r=2$ would be enough due to the saturation
  at the function~$\Theta^{2}$.
\end{rem}

We exemplify the above bounds for the bias for power type behavior,
both of the smoothness in terms of~$\varphi(t) = t^{\beta}$, and the
linking function~$\psi(t) = t ^{\kappa}$. This results in a
function~$\Theta^{2}(t) = t^{1 + 2\kappa}$, which has operator concave
inverse. Thus, this requirement in Assumption~\ref{ass:link} is
fulfilled whatever~$\kappa>0$ is found.

Then, the low order case~$0 \prec \varphi\prec \varphi_{0}$ covers
the range~$0 < \beta \leq 1/2$, and in this range the
function~$\varphi^{2}(t) = t^{2\beta}$ is operator concave,
because~$2\beta \leq 1$. 

The regular case~$\varphi_{0} \prec \varphi \prec \Theta$ covers the
exponents~$1/2 \leq \beta\leq 1/2 + \kappa$, since the operator concavity was
assumed to hold for the function
$$
\lr{\frac{\varphi}{\varphi_{0}}}^{2}\lr{\lr{\Theta^{2}}^{-1}(t)} 
= t^{\frac{2(\beta -1/2)}{1 + 2\kappa}}.
$$

Finally, it is seen similarly, that the high order case covers the
range~$1/2 \leq \beta \leq 1/2 + (\rpar/2) (1 + 2 \kappa)$, which for $\rpar=2$
already is beyond the saturation point~\peter{$1 + 2\kappa$}.
 
\section{Bounding the posterior spread}
\label{sec:spread-bound}

We recall the structure of the posterior spread from~(\ref{eq:spread})
as
$$
  \tr{\Cad} = \delta^{2} \tr{\lr{\alpha\op I + \bastb}^{-1}\prc}. 
$$
As can be seen, the noise level~$\delta$ enters quadratically, and we
aim at finding the dependence upon the scaling parameter~$\alpha$. To
this end the following result proves to be useful.
\begin{prop}\label{pro:spread-bound}
Under Assumption~\ref{ass:link} we have that
$$
\tr{\Cad}\leq \frac {\delta^{2}} {m^{2}} \tr{\lr{\alpha\op I +
    \bastb}^{-1}f_{0}^{2}\lr{\bastb}}.
$$ 
\end{prop}
\begin{proof}
We start with the situation as given in~(\ref{eq:prc-upper-bound}). 
This order extends by multiplying $\lr{\alpha\op I +
  \bastb}^{-1/2}$ from both sides, such that we conclude that
$$
\lr{\alpha\op I + \bastb}^{-1/2} \prc\lr{\alpha\op I +   \bastb}^{-1/2} 
\leq \frac 1 {m^{2}} \lr{\alpha\op I + \bastb}^{-1/2}  f_{0}^{2}(\bastb)
\lr{\alpha\op I +   \bastb}^{-1/2} .
$$
Now we apply the Weyl Monotonicity Theorem, see
e.g.~\cite[Cor.~III.2.3]{MR1477662} to see that this inequality
applies to all singular numbers. But the operators on both sides are
self-adjoint and positive, such that singular numbers and eigenvalues
coincide. Thus we arrive at
\begin{align*}
  \tr{\Cad} & = {\delta^{2}}\tr{\lr{\alpha\op I + \bastb}^{-1}\prc} 
= {\delta^{2}}\tr{\lr{\alpha\op I + \bastb}^{-1/2}\prc\lr{\alpha\op I + \bastb}^{-1/2} }\\
& \leq \frac {\delta^{2}} {m^{2}} \tr{\lr{\alpha\op I + \bastb}^{-1/2}  f_{0}^{2}(\bastb)
\lr{\alpha\op I +   \bastb}^{-1/2} }\\
& = \frac {\delta^{2}} {m^{2}}\tr{\lr{\alpha\op I + \bastb}^{-1}  f_{0}^{2}(\bastb)},
\end{align*}
where we used the cyclic commutativity of the trace. The proof is complete.
\end{proof}

\section{Bounding the squared posterior contraction}
\label{sec:spc-bound}

In the previous sections we derived bounds for both the bias and the
posterior spread. In all the smoothness cases from
Section~\ref{sec:bias-bound} we arrived at a bound of the following
form. If $\xast$ has smoothness with index function~$\varphi$, and if
the link condition is with operator concave function~$f_{0}^{2}$
from~(\ref{eq:f-interpolation}) then it was shown in
Propositions~\ref{pro:regular-bound}--~\ref{pro:high-smoothness} that
\begin{equation}
  \label{eq:bias-bound-summary}
\bx(\alpha) \leq \frac M m \norm{\sa(\bastb)
  \varphi\lr{f_{0}^{2}(\bastb)}},\quad \alpha>0.  
\end{equation}
Also, the posterior spread was bounded in
Proposition~\ref{pro:spread-bound} as
$$
\tr{\Cad} \leq \frac {\delta^{2}} {m^{2}} \tr{\lr{\alpha\op I +
    \bastb}^{-1}f_{0}^{2}\lr{\bastb}},\quad \alpha >0.
$$ 
As was discussed in Remark~\ref{rem:saturation} we shall confine to the
case when~$\varphi \prec \Theta^{2}$, i.e.,\ before the saturation point.
 If this
is the case then we can bound, by using that~the function~$\sa$
obeys~$\sa(t) t \leq \alpha,\ t,\alpha>0$, the bias by
$$
\bx(\alpha) \leq \frac M m \varphi \lr{f_{0}^{2}(\alpha)},\quad \alpha>0.
$$
A similar 'handy' explicit bound for the posterior spread can hardly
be given. Under additional assumptions on the decay rate of the
singular numbers  more explicit bounds can be given. We refer
to~\cite[Sect.~4]{MR3406424}, in particular Assumption~5 and Lemma~4.2 ibid. for details.

Overall we obtain the following result.
\begin{thm}
  [Bound for the~$\spc$]\label{thm:spc}
Suppose that assumptions~\ref{ass:link} \&~\ref{ass:smoothness} hold
for index functions~$\varphi$ and~$f_{0}$. Under the assumptions of
Propositions~\ref{pro:regular-bound}, \ref{pro:low-order}
\&~\ref{pro:high-smoothness}, respectively, and if~$\varphi \prec \Theta^{2}$, then
\begin{equation}
  \label{eq:spc-overall}
  \spc(\alpha,\delta) \leq \frac{M^{2}}{m^{2}} \inf_{\alpha>0}\left[\varphi^{2}
  \lr{f_{0}^{2}(\alpha)} + 2 \tr{\lr{\alpha\op I +
    \bastb}^{-1}f_{0}^{2}\lr{\bastb}}\right],\  \alpha,\delta>0.
\end{equation}
\end{thm}

The above analysis is given in abstract terms of index functions, and
it is worthwhile to give an example to compare this with
known (and minimax)  bounds for the commuting case.

To this end we treat the case for a moderately ill-posed
operator~$\prc$, a power type link, and Sobolev type smoothness, with
parameters~$a,p>0$ and~$\beta>0$ as in the original
studies~\cite{MR2906881,AM2014}. 
\begin{xmpl}[power type decay]\label{ass:xmpl}\ \\
  \begin{enumerate}
  \item  For some~$a>0$ we have
    that~$s_{j}(\prc) \asymp j^{-(1+2a)},\ j=1,2,\dots$.
  \item There is some $p>0$ such that~$\Theta^{2}(t) \asymp t^{\frac{1
        + 2a + 2p}{1 + 2a}}$ as~$t\to 0$, and\label{ali:theta2}
  \item There is some~$R<\infty$ such that~$\sum_{j=^{\infty}}
    j^{2\beta}\lr{\xast_{j}}^{2}\leq R^{2}$, where
    ~$\xast_{j},\ j=1,2,\dots$ denote the coefficients
    of~$\xast$ with respect to the eigenbasis of~$\prc$.\label{ali:smooth}
  \end{enumerate}  
\end{xmpl}
This gives for the compound operator~$\bastb$ that
$$
s_{j}(\bastb) \asymp s_{j}\lr{\Theta^{2}(\prc)} =
\Theta^{2}(s_{j}(\prc))\asymp j^{-(1 + 2a + 2p)},\quad j=1,2,\dots
$$
Notice furthermore that
$$
s_{j}\lr{f_{0}^{2}(\bastb)} = f_{0}^{2}\lr{s_{j}(\bastb)} \asymp
j^{-(1 + 2a)},\quad j=1,2\dots
$$
Then we can bound, and we omit the standard calculations,  the posterior spread by using
Proposition~\ref{pro:spread-bound} as
$$
\tr{\Cad} \leq \delta^{2}\sum_{j=1}^{\infty}
\frac{s_{j}(f_{0}^{2}(\bastb)}{\alpha + s_{j}(\bastb)}
\asymp \delta^{2}\alpha^{-\frac{1 + 2p}{1 + 2a + 2p}}.
$$
We turn to the description of the smoothness of~$\xast$ in terms of an
index function~$\varphi$, thus rewriting the \peter{condition (\ref{ali:smooth})}. This yields
that~$\varphi(t) = t^{\frac{\beta}{1 + 2a}}$, see Section~4
from~\cite{AM2014} for details. 
We see from \peter{condition~(\ref{ali:theta2})} that saturation is at~$\beta= 1 + 2a + 2p$.

Thus for~$0 < \beta \leq 1 + 2a + 2p $ we apply
the bias bound from~(\ref{eq:bias-bound-summary}) for obtaining a tight bound for
the~$\spc$. We balance the squared bias with the bound for the
posterior spread as~$\alpha^{\frac{2\beta}{1 + 2a + 2p}}=
\delta^{2}\alpha^{-\frac{1 + 2p}{1 + 2a + 2p}}$. 

This gives~$\aast =
\left[\delta^{2}\right]^{\frac{\beta}{1 + 2\beta + 2p}}$, and finally
this results in rate for the decay of the~$\spc$ as
$$
\spc(\aast(\delta),\delta) = \bigo
\lr{\left[\delta^{2}\right]^{\frac{2\beta}{1 + 2\beta + 2p}}} \quad
\text{as}\  \delta\to 0
$$
if~$\beta \leq 1 + 2a + 2p$. Such bound is well known for commuting
operators, see~\cite[\S~4.1]{AM2014}, and the original
study~\cite[Thm.~4.1]{MR2906881}. 

Next we sketch the way to obtain bounds for the backwards heat
equation, with an  exponentially ill-posed operator.
\begin{xmpl}
  [backwards heat equation, {cf.~\cite{MR3031282}}]\ \\
  \begin{enumerate}
  \item  For some~$a>0$ we have
    that~$s_{j}(\prc) \asymp j^{-(1+2a)},\ j=1,2,\dots$.
  \item The linking function~$\Theta^{2}$ obeys~$\Theta^{2}(t) = e^{-2
      t^{\frac{2}{1 + a}}}$.
  \item There is some~$\beta >0$ such that~$\varphi(t) = t^{\beta/(1 +
      2a)},\ t>0$.
  \end{enumerate}
First, the smoothness assumption is as in the  previous example.
In this case always~$\varphi \prec \Theta$, such that
 there is no saturation.

For bounding the bias we see that~$f_{0}^{2}(t) = \frac 1 2 \log^{-(1 +
  2a)}(1/t),\ t<1$. For~$\beta/(1+ 2a) \leq 1/2$ the position will
thus be
$$
t \longrightarrow \left[\frac 1 2\log^{-(1+ 2 a)}(1/t) \right]^{2\beta/(1  + 2a)}
= 4^{- \beta/(1  + 2a)} \log^{-2\beta}(1/t),\quad \text{as} \ t\to 0,
$$
and this is operator concave for~$0 < \beta \leq 1/2$. In the regular
case, a similar calculation reveals that the function
$$
t \longrightarrow 4^{(\beta-1/2)} \log^{-2\beta + 1}(1/t),
$$
must be operator concave, which is true for~$1/2 \leq \beta \leq 1$.

So, in the range~$0 < \beta \leq 1$  we find that
$$
\brf= \bigo\lr{\log^{-\beta/2}(1/\alpha)}\quad \text{as}\ \alpha\to 0.
$$
By standard calculations we bound the posterior spread as
$$
\tr{\Cad} \leq C \delta^{2}\frac 1 \alpha\log^{-a}(1/\alpha),
$$
for some constant~$C<\infty$. Applying Theorem~\ref{thm:spc} we
let~$\aast(\delta) := \delta^{2} \log^{\beta - a}(1/\alpha)$ and get the rate
$$
\spc(\aast(\delta),\delta) = \bigo\lr{\log^{-\beta}(1/\delta)}\quad
\text{as}\ \delta\to 0.
$$
This corresponds to the contraction rate of the posterior as presented
in~\cite{MR3031282} with~$\delta\sim n^{-1/2}$. For details we refer
to Section~4 of the survey~\cite{AM2014}. However, while these results
cover all~$\beta >0$, the non-commuting case will cover only the
range~$0 < \beta \leq 1$.
\end{xmpl}

\section{Conclusion}
\label{sec:conclusion}

We summarize the above findings, and we start with the bias bounds.  In either of the three cases, if
there is a valid link condition, if smoothness is given as in
Assumption~\ref{ass:smoothness}, and if the involved functions are
operator concave, then the norm in~$\brf$ from~(\ref{eq:bias-form})
can be bounded by 
$$
\brf \leq \frac M m \norm{\funf(\bastb)\varphi(f_{0}^{2}(\bastb))}{}.
$$
This seems to be the natural extension for the bias bound to
the non-commuting context. In the commuting context we would get
$f_{0}^{2}(\bastb) = \prc$, see
\S~\ref{sec:commute}.

Under these premises the analysis from~\cite{AM2014} can be extended to
the non-commuting situation. We stressed in
Remark~\ref{rem:saturation} that a lifting of the original link
condition is necessary in order to yield optimal order bounds for the
squared posterior contraction up to the saturation point.

The analysis from~\cite{MR3084161} covers
by different techniques the regular case. In case of a power type
function~$\psi$, and hence of~$\Theta$, the requirements of operator concavity reduce to
power type functions with power in the range between $(0,1)$, and
hence these are automatically fulfilled.

For the posterior spread we derived a similar extension to the
non-commuting case in Proposition~\ref{pro:spread-bound}. There is no
handy way to derive the exact increase of the spread as~$\alpha\to
0$. Under additional assumption on the regularity of the decay for the
singular numbers of~$\bastb$ this problem can be reduced to the
\emph{effective dimension} of the operator~$\bastb$, given
as~$\N_{\bastb}(\alpha) = \tr{\lr{\alpha\op I + \bastb}^{-1}\bastb}$. We
did not pursue this line, here. Instead we refer to the
study~\cite{MR3406424}.

Finally, we presented two examples exhibiting the obtained rates for
the $\spc$, both for moderately and severely ill-posed operators. More
examples, using functional dependence for commuting operators,  are given in the study~\cite{AM2014}.

\bibliographystyle{plain}
\bibliography{iterate}

\end{document}